\documentclass[12pt]{amsart}
\usepackage[top=1in, bottom=1in, left=1in, right=1in]{geometry}
\usepackage{times, amsthm, amssymb, amsmath, amsfonts, amsthm, bm, graphicx, mathrsfs}
\usepackage{latexsym}
\usepackage{amscd}
\usepackage{float}
\usepackage[all]{xy}
\usepackage[usenames,dvipsnames]{color}
\usepackage{tikz}
\usepackage{graphicx}
\usepackage{epstopdf}
\usepackage[colorlinks,
pdftex]{hyperref}

\newtheorem{thm}{Theorem}[section]
\newtheorem{pro}[thm]{Proposition}
\newtheorem{lem}[thm]{Lemma}
\newtheorem{cor}[thm]{Corollary}

\newtheorem{dfn}[thm]{Definition}

\theoremstyle{definition}
\newtheorem{example}[thm]{Example}
\newtheorem{remark}[thm]{Remark}

\newtheorem{convention}[thm]{Convention}

\DeclareMathAlphabet{\mathpzc}{OT1}{pzc}{m}{it}   
\usepackage{calligra}
\DeclareMathAlphabet{\mathcalligra}{T1}{calligra}{m}{n}

\def\<{\langle}
\def\>{\rangle}
\def\0{\mathbf{0}}
\def\CC{{\mathbb C}}
\def\NN{{\mathbb N}}
\def\RR{{\mathbb R}}
\def\TT{{\mathbb T}}
\def\ZZ{{\mathbb Z}}
\def\cA{{\mathcal A}}

\newcommand{\Log}{\operatorname{Log}}

\newcommand{\Arg}{\operatorname{Arg}}

\def\Re{{\rm Re\, }}

\def\bee{\begin{equation*}}
\def\eee{\end{equation*}}

\def\be{\begin{equation}}
\def\ee{\end{equation}}

\def\inter{{\rm int}}

\def\supp{{\rm supp}}
\def\vol{{\rm vol}}
\def\bZ{\mathcal{Z}}
\def\rI{\text{I}}
\def\rII{\text{I\hspace{-1pt}I}}
\def\cD{\mathcal{LA}'}
\def\cM{\mathcal{M}}

\numberwithin{equation}{section}
\begin{document}
\vspace*{-3mm}
\mbox{}
\title{Euler--Mellin integrals and $A$-hypergeometric functions}

\author{Christine Berkesch, Jens Forsg{\aa}rd, and Mikael Passare}
\address{Department of Mathematics \\ Stockholm University \\
SE-106 91 Stockholm, Sweden.}
\email{cberkesc@math.su.se, jensf@math.su.se, passare@math.su.se}
\thanks{The first author was supported by NSF Grant OISE 0964985.}

\begin{abstract}
We consider integrals that generalize both the Mellin transforms of rational functions of the form $1/f$ and the classical Euler integrals. The domains of integration of our so-called Euler--Mellin integrals are naturally related to the coamoeba of $f$, and the components of the complement of the closure of the coamoeba give rise to a family of these integrals. After performing an explicit meromorphic continuation of Euler--Mellin integrals, we interpret them as $A$-hypergeometric functions and discuss their linear independence and relation to Mellin--Barnes integrals.
\end{abstract}

\maketitle
\mbox{}
\parskip=1ex
\parindent0pt
\vspace{-6mm}
\section{Introduction}
\label{sec:intro}

In the classical theory of hypergeometric functions, 
a prominent role is played by the Euler integral formula
\[
\medskip
{}_2F_1(s; t; u) = 
\frac{\Gamma(t)}{\Gamma(s_1)\Gamma(s_2)}
\int_0^1 x^{s_1-1}(1-x)^{t-s_1-1}(1-ux)^{-s_2}dx,
\smallskip
\] 
which yields an analytic continuation of the Gauss hypergeometric series 
${}_2F_1$ from the unit disk $|u|<1$ to the larger domain $\big|\!\arg(1-u)\big|<\pi$.
However, this Euler integral is not symmetric in $s_1$ and $s_2$, 
even though the function ${}_2F_1$ enjoys such symmetry.
Following Erd\'elyi \cite{E}, one can introduce another variable of integration 
and obtain the symmetric formula
\medskip
\be\label{eq:sym 2F1}
{}_2F_1(s; t; u) =
G(s,t) \int_0^1\!\!
\int_0^1 \!\!x^{s_1-1}y^{s_2-1}(1-x)^{t-s_1-1}(1-y)^{t-s_2-1}(1-uxy)^{-t}dx\wedge dy,
\ee
where 
\[
G(s,t)= \Gamma(t)^2 \,/\, (\Gamma(s_1)\Gamma(s_2)\Gamma(t-s_1)\Gamma(t-s_2)).
\]
After making the substitutions $z=x/(1-x)$, $w=y/(1-y)$, and $c=1-u$, 
one finds that the double integral in~\eqref{eq:sym 2F1} takes the simple form
\medskip 
\be\label{eq:double int intro}
\int_0^\infty\!\!\!\!
\int_0^\infty \!\frac{z^{s_1}w^{s_2}}{(1+z+w+czw)^t}\frac{dz\wedge dw}{zw},
\medskip
\ee
which restricted to $t=-1$ is a twofold Mellin transform of $1/f$, where $f(z,w)= 1+z+w+czw$. 
In this paper we introduce a generalization of the Mellin transform of a rational function $1/f$, which we call an Euler--Mellin integral. The general form of an Euler--Mellin integral can be found in~\eqref{eq:eulermellin} below.

Euler--Mellin integrals are closely related to $A$-hypergeometric 
Euler type integrals (see~\cite{GKZc,SST}). 
The most notable difference between these previously studied functions 
and the Euler--Mellin integrals we introduce here is the domain of integration. 
We integrate over explicit, simply connected, but non-compact sets, 
whereas previous authors used compact yet rather elusive cycles. 
We show that the simple connectivity of our domain of integration 
allows us to handle the multivaluedness of the integrand; 
however, to achieve convergence, the non-compactness initially restricts the values of  the parameters $(s,t)$, 
see Theorem~\ref{thm:expgro}. 
In Theorem~\ref{thm:mel}, we remove the restrictions on $(s,t)$ through  
an explicit meromorphic continuation of an  Euler--Mellin integral. 
This yields a meromorphic function whose singular locus is contained in certain families of hyperplanes. Taking these into account, we obtain a function which is entire in the parameters $(s,t)$. 

A further generalization of the results of \S\ref{sec:em integrals} is achieved by considering the coamoeba of the polynomial $f$, as defined in~\eqref{eq:def coamoeba}. In practice, we rotate the domain of integration of the Euler--Mellin integral to 
$\Arg^{-1}(\theta) = \{ z\in (\CC_*)^n  \mid \Arg(z) = \theta \}$,
for appropriate choices of $\theta$ (see \S\ref{sec:coamoebas}).
In the previous example, this means replacing the integral \eqref{eq:double int intro} by  
\[ 
\int_{\Arg^{-1}(\theta)} \!\frac{z^{s_1}w^{s_2}}{(c_1+c_2z+c_3w+c_4zw)^t}\frac{dz\wedge dw}{zw}. 
\] 

The $A$-hypergeometric approach considers a polynomial $f$ with general coefficients on a fixed set of monomials, which are identified with a matrix $A$. 
We show in Theorem~\ref{thm:gkz} that for a generic choice of coefficients of $f$, the corresponding Euler--Mellin integral with parameters $(s,t)$ satisfies an $A$-hypergeometric system of differential equations with parameter $\beta = -(t,s)$. 
In particular, the meromorphic continuations of Euler--Mellin integrals obtained through \S\S\ref{sec:em integrals}--\ref{sec:coamoebas} provide a family of $A$-hypergeometric functions that are entire in the parameter $\beta$ (see Definition~\ref{def:gkz}). 

A key open problem in the study of an $A$-hypergeometric system is to explain how its solution space varies with the parameter $\beta$. 
After meromorphic continuation, Euler--Mellin integrals provide a new tool to approach this task. As illustrated in \S\ref{sec:gkz}, the explicit nature of the domain of integration in an Euler--Mellin integral makes it possible to perform computations in examples where the variation of $A$-hypergeometric solutions is not yet fully understood. 
We also show in Proposition~\ref{prop:indep} that under certain conditions on $A$, Euler--Mellin integrals provide a full basis of solutions for generic $\beta$. 
An example in \S\ref{sec:0134} provides  evidence that Euler--Mellin integrals explain the appearance of special solutions at rank-jumping parameters.

Mellin--Barnes integrals are another class of $A$-hypergeometric integrals previously considered in the literature~\cite{N,beukers-monodromy} (see Definition~\ref{def:MB}); in particular, such integrals are used by Beukers to compute elements in the local monodromy group of an $A$-hypergeometric system. In \S\ref{sec:MellinBarnes}, we compare these to Euler--Mellin integrals. A consequence of Theorem~\ref{thm:EM-MB} is that there are at least as many linearly independent Euler--Mellin integrals as Mellin--Barnes integrals; this provides additional examples in which Euler--Mellin integrals are linearly independent at generic $\beta$ (see Corollary~\ref{cor:LopsidedIntegralIndependence}).  

\vspace{-3mm}
\subsection*{Outline}
In \S\ref{sec:em integrals}, we introduce Euler--Mellin integrals, show their convergence, and perform their meromorphic continuation, which is our main result. 
In \S\ref{sec:coamoebas}, we employ coamoebas to extend the results of the previous section to include more general domains of integration. 
Euler--Mellin integrals are shown to be $A$-hypergeometric functions in \S\ref{sec:gkz}, providing a basis of solutions to $A$-hypergeometric systems in certain cases. 
In \S\ref{sec:MellinBarnes}, we relate Euler--Mellin integrals to Mellin--Barnes integrals, obtaining further insight into the linear independence of both sets of integrals. Finally, \S\ref{sec:0134} contains an example of the behavior of Euler--Mellin integrals at a rank-jumping parameter of an $A$-hypergeometric system. 

\vspace{-3mm}
\subsection*{Acknowledgements}
This work was done in connection with the program ``Algebraic geometry with a view toward applications" at Institut Mittag-Leffler, and the authors are thankful to the Institute for its hospitality. 
We are also thankful to Alicia~Dickenstein and Laura~Felicia~Matusevich for helpful conversations. 

\section{Convergence and meromorphic continuation of Euler--Mellin integrals}
\label{sec:em integrals}

This section contains our main result, Theorem~\ref{thm:mel}, which provides an explicit presentation of a meromorphic continuation of the Euler--Mellin integral of a polynomial $f$ of several variables $z_1,z_2,\dots,z_n$. 
Given a polynomial $f = \sum_{\alpha\in\supp(f)} c_\alpha z^\alpha$, 
the \emph{Euler--Mellin integral} is a natural generalization of the Mellin transform of a rational function $1/f$ of several variables and is given by
\be\label{eq:eulermellin} 
M_{f}(s,t)
:=\int_{\RR_+^n} \frac{z^s}{f(z)^t}
	\frac{dz_1\wedge\ldots\wedge dz_n}{z_1\cdots z_n}
=\int_{\RR^n} \frac{e^{\< s,x\>}}{f(e^x)^{t}}\,
	dx_1\wedge\ldots\wedge dx_n,
\smallskip
\ee
where $\RR^n_+:=(0,\infty)^n$ denotes the positive orthant in $\RR^n$. 
Here we employ multi-index notation for variables $z_1,\dots, z_n$ 
and polynomials $f_1, \dots, f_m$; that is, for $s\in\CC^n$ and $t\in\CC^m$, 
we write $z^s := z_1^{s_1}\cdots z_n^{s_n}$ and $f(z)^t := f_1(z)^{t_1}\cdots f_m(z)^{t_m}$. 
Whenever there is no risk of confusion, we use the notation 
$f(z) := f(z)^{(1,\dots,1)} = \prod_{i=1}^m f_i$. 

In order for such an integral to converge, 
restrictions must a priori be placed on both the exponent vector $(s,t)$ 
and the polynomial $f$; it is not
enough to demand only that each $f_i$ is nonvanishing on $\RR^n_+$.
We next provide such a domain of convergence for the 
Euler--Mellin integral~\eqref{eq:eulermellin}, 
generalizing \cite[Thm.~1]{NP}.

\begin{dfn}\label{def:nonvng}
If $\Gamma$ is a face of the Newton polytope $\Delta_f$ of $f$, then the \emph{truncated polynomial with support $\Gamma$} is given by 
\[
f_\Gamma := \sum_{\alpha\in \Gamma\cap \supp(f)} c_\alpha z^\alpha.
\]
The polynomial $f$ is said to be \emph{completely nonvanishing} 
on a set $X$ if for each face $\Gamma$ of $\Delta_f$, 
the truncated polynomial $f_\Gamma$ has no zeros on $X$. 
In particular, the polynomial $f$ itself does not vanish on $X$.
\end{dfn}

For a vector $\tau\in\RR_{+}^m$, we denote by 
$\tau\Delta_f$ the weighted Minkowski sum 
$\sum_{i=1}^m \tau_i \Delta_{f_i}$ of the Newton polytopes of the $f_i$ with respect to $\tau$. 
Note that with this notation, 
the Newton polytope of $f$ satisfies $\Delta_f = (1,\dots,1)\Delta_f$. 

\begin{thm}\label{thm:expgro}
If each of the polynomials $f_1,\dots, f_m$ are completely nonvanishing 
on the positive orthant $\RR^n_+$ (as in Definition~\ref{def:nonvng}), then the Euler--Mellin integral $M_f(s,t)$ of~\eqref{eq:eulermellin} 
converges and defines an analytic function in the tube domain 
\[
\bigl\{(s,t)\in\CC^{n+m}\mid 
	\tau:= \Re t \in\RR_+^m,\ \sigma := \Re s\in\inter(\tau\Delta_f) \bigr\}.
\] 
\end{thm}
\begin{proof}
It suffices to prove that for any $(s,t)$ with all $\tau_i >0$ and 
$\sigma\in \inter(\tau\Delta_f)$, 
there exist positive constants $c$ and $k$ such that
\[
\big|f(e^x)^t e^{-\<s,x\>}\big| = \big|f(e^x)^{t}\big| e^{-\<\sigma,x\>} \ge c e^{k|x|}
\qquad\text{for all}\quad x\in\RR^n.
\]
In fact, it is enough to show that this inequality holds outside of some ball 
$B(0)$ in $\RR^n$.

Since $\sigma\in\inter(\tau\Delta_f)$, we can expand it as a sum 
$\sigma = \sigma_1 + \dots + \sigma_m$ of $m$ vectors such that 
$\sigma_i/\tau_i \in \inter(\Delta_{f_i})$. 
It is shown in the proof of \cite[Thm.~1]{NP} that for each 
$\sigma_i \in \inter(\Delta_{f_i})$ there are positive constants $c_i$ and $k_i$ 
such that for $x$ outside of some ball $B_i(0)$, 
\[
\big|f_i(e^x)\big| e^{-\<\sigma_i,x\>} \ge c_i e^{k_i|x|}.
\] 
Note that it is essential in \cite[Thm.~1]{NP} 
that $f_i$ is completely nonvanishing on the positive orthant. 
Thus for $x$ outside of $B(0) = \bigcup_{i=0}^m B_i(0)$, we have 
\be\label{eq:bound}
\big|f(e^x)^{t}\big| e^{-\<\sigma,x\>} 
\ = \ 
\prod_{i=1}^m\Big(\big|f_i(e^x)\big| e^{-\left\<\frac{\sigma_i}{\tau_i},x\right\>}\Big)^{\tau_i}
\ \geq \ 
\prod_{i=1}^mc_i^{\tau_i}e^{\tau_i k_i|x|} 
\ = \ 
ce^{k|x|},
\ee
where $c = c_1^{\tau_1}\cdots c_m^{\tau_m}$ and 
$k = \tau_1 k_1+\dots+\tau_m k_m$ are the desired positive constants. 
\end{proof}

\begin{example}\label{ex:2F1classical}
By a classical integral representation of the Gauss hypergeometric function 
${}_2 F_1$, 
\smallskip
\be\label{eq:classical 2F1}
\int_0^\infty
\frac{z^s}{(1 + z)^{t_1}(c+z)^{t_2}} \frac{dz}{z}
\ =\ 
\frac{\Gamma(t_1+t_2-s) \Gamma(s)}{\Gamma(t_1+t_2)}
\ {}_2 F_1(t_2,t_1+t_2-s;t_1+t_2;1-c)
\smallskip
\ee
for $\Re (t_1+t_2) > \Re (t_1+t_2-s) > 0$ and $|\arg(c)| < \pi$, where $\arg$ denotes the principal branch of the argument mapping. 
Note that $|\arg(c)| < \pi$ is equivalent to $f(z) = (1 + z)(c+z)$ 
being completely nonvanishing on $\RR_+$. 
Since $\Delta_{f_1} = \Delta_{f_2} = [0,1]$, the condition that 
$\sigma \in \inter(\tau \Delta_f)$ is the same as $0<\Re(s)<\Re(t_1+t_2)$. 
We also note that the right hand side of~\eqref{eq:classical 2F1} 
is analytic in this domain. Further, since $\Re(t_1) > 0$ and $\Re(t_2)>0$, 
the convergence domain given in Theorem \ref{thm:expgro} is not optimal; 
however, being full-dimensional, 
it is large enough for our goal of meromorphic continuation.
\end{example}

As the right hand side of \eqref{eq:classical 2F1} is a meromorphic function 
in $s$ and $t$, it provides a meromorphic extension of the corresponding Euler--Mellin integral. 
On this right side, we have the regularized ${}_2F_1$ as one factor, thus the 
polar locus of the meromorphic extension is contained in two families of hyperplanes given 
by the polar loci of the Gamma functions. 
Our main result shows that this kind of meromorphic 
continuation is possible for all Euler--Mellin integrals. 

To obtain the strongest form of this result, 
we choose a specific presentation for $\tau\Delta_f$. 
To begin, each Newton polytope $\Delta_{f_i}$ can be written uniquely 
as the intersection of a finite number of halfspaces
\be\label{eq:newton fi}
\Delta_{f_i}=\bigcap_{j=1}^{N_i}
\bigl\{\sigma\in\RR^n \mid \<\mu^i_j,\sigma\>\geq\nu^i_j\bigr\}, 
\ee
where the $\mu^i_j$ are primitive vectors. 
Fixing an order, let $\{\mu_1,\dots,\mu_N\}$ be equal to the set 
$\{\mu^i_j\mid 1\le i\le m, \le j\le N_i\}$, where we assume that $\mu_i
\neq\mu_j$ for all $i\neq j$. 
We now extend the definitions of $\nu_j^i$ from~\eqref{eq:newton fi} to each $\mu_k$; 
namely, for each $k$, let $\nu_k := (\nu_k^1,\dots,\nu_k^m)$ with 
\[
\nu^i_k := \min\{ \<\mu_k,\alpha\> \mid \alpha\in\Delta_{f_i}\}, 
\] 
and set $|\nu_k| := \nu_k^1 + \dots + \nu_k^m$. 
It now follows from the definition of the $\nu_k$ that
\be
\label{eq:newtonpol2}
\tau\Delta_f = \bigcap_{k=1}^N
	\bigl\{\sigma\in\RR^n \mid \<\mu_k,\sigma\>\geq\<\nu_k,\tau\>\bigr\} 
\ee
and 
$\inter (\tau\Delta_f) = \sum_{i=1}^m \tau_i \,\inter(\Delta_{f_i})$. 

We are now prepared to state our main result, which provides a meromorphic continuation of~\eqref{eq:eulermellin}, generalizing \cite[Thm.~2]{NP}. 
In \S\ref{sec:coamoebas}, we obtain a stronger form of 
the result by relaxing the condition that the $f_i$ be completely nonvanishing 
on $\RR^n_+$.

\begin{thm}\label{thm:mel}
If the polynomials $f_1,\dots, f_m$ are completely nonvanishing on 
the positive orthant $\RR^n_+$ (as in Definition~\ref{def:nonvng}) and 
the Newton polytope $\Delta_{f_1\cdots f_m}$ is of full dimension $n$,
then the Euler--Mellin integral $M_{f}(s,t)$ 
admits a meromorphic continuation of the form
\begin{equation}\label{eq:mero}
M_{f}(s,t)=\Phi_f(s,t)\prod_{k=1}^N \Gamma(\<\mu_k,s\>-\<\nu_k,t\>),
\end{equation}
where $\Phi_f(s,t)$ is an entire function and $\mu_k,\nu_k$ 
are given by~\eqref{eq:newtonpol2}.
\end{thm}
\begin{proof}
By Theorem~\ref{thm:expgro}, the original Euler--Mellin integral $M_f(s,t)$ of~\eqref{eq:eulermellin} 
converges on 
\[
\left\{ (s,t)\in\CC^{n+m} \,\big|\,
\tau:=\Re(t)\in\RR_{+}^m,\ 
\sigma:=\Re(s) \text{ such that } \<\mu_k,\sigma\>>\<\nu_k,\tau\> 
\text{ for all } 1\leq k \leq N
\right\}, 
\]
which is a domain since $\Delta_f$ is of full dimension. 
Our goal is to expand the convergence domain of 
the integral~\eqref{eq:eulermellin}, 
at the cost of multiplication by terms corresponding to 
the poles of the Gamma functions appearing in~\eqref{eq:mero}. 
We do this iteratively, integrating by parts in the direction of 
a vector $\mu_k$ at each step. This expands the domain of 
convergence in the opposite direction of $\mu_k$ by a distance $d_k$, 
which we determine explicitly. 

To begin, we set notation for the first iteration in one direction. 
Fix $k$ between $1$ and $N$, and let $\Gamma$ be the face of $\Delta_{f_i}$ 
corresponding to $\mu_k$ and $\nu_k$. 
For $\alpha\in\supp(f)$, set 
\[
d_k^\alpha := \<\mu_k, \alpha\> - |\nu_k|. 
\]
Since $\alpha\in\Delta_f$, it follows that $d_k^\alpha \geq 0$. 
In particular, since there is a decomposition 
$\alpha = \sum_i \alpha_i$ with $\alpha_i \in \Delta_{f_i}$, we see that 
$d_k^\alpha = 0$ if and only if $\<\mu_k, \alpha_i\> = \nu_k^i$ for all $i$.

Fixing $i$, the truncated polynomial $(f_i)_\Gamma$ has the homogeneity 
$(f_i)_\Gamma(\lambda^{\mu_k}z )=\lambda^{\nu^i_k}(f_i)_\Gamma(z)$, where $\lambda$ is any nonzero complex number and $\lambda^{\mu_k}z = (\lambda^{\mu_k^1}z_1,\lambda^{\mu_k^2}z_2,\dots,\lambda^{\mu_k^n}z_n)$. 
Hence the coefficients of the scaled polynomial $\lambda^{-\nu^i_k}{(f_i)_\Gamma}(\lambda^{\mu_k}z)$ are independent of $k$ and the $\lambda$.  
In particular, we have that the Newton polytope of 
\[
f'_i(z):= \frac{d}{d\lambda}
\Bigl(\lambda^{-\nu^i_k}{f_i}(\lambda^{\mu_k}z)\Bigr)\bigg|_{\lambda=1}
\]
is disjoint from $\Gamma$.
This fact allows us to extend the domain of convergence 
of~\eqref{eq:eulermellin} over the hyperplane defined by 
$\<\mu_k,\sigma\>=\<\nu_k,\tau\>$ as follows. 
Since $M_{f}(s,t)$ is independent of $\lambda$, we have 
\[
0
=\frac{d}{d\lambda}\int_{\RR_+^n}
\frac{(\lambda^{\mu_k}z)^{s}}{{f(\lambda^{\mu_k}z)^{t}}}\frac{dz}{z}
=\frac{d}{d\lambda}\left[\lambda^{\<{\mu_k},s\>-\<\nu_k,t\>}
\int_{\RR_+^n}
\frac{z^{s}}{{\lambda^{-\<\nu_k,t\>}f(\lambda^{\mu_k}z)^{t}}}
\frac{dz}{z}\right].
\]
Thus differentiating~\eqref{eq:eulermellin} with respect to $\lambda$ 
and setting $\lambda=1$ yields the identity
\begin{equation}\label{eq:ett}
M_{f}(s,t) = \frac{1}{\<\mu_k,s\>-\<\nu_k,t\>}
\int_{\RR_+^n}\frac{z^{s} g_{k}(z)}{f(z)^{t+1}}\frac{dz}{z},
\end{equation}
where $g_k$ is the polynomial
\[
g_k = \sum_{i=1}^m t_i \cdot f_1\cdots {f'_i}\cdots f_m.
\]
Note that $\supp(g_k)$ is contained in $\supp(f)$; 
moreover, since $\Gamma$ is the face of $\Delta_f$ 
corresponding to $\mu_k$ and 
$\supp(f'_i)$ is disjoint from $\Delta_{f_i}\cap \Gamma$, 
we see that $\supp(g_k)$ is disjoint from $\Gamma$. 
In other words, $d_k^\alpha > 0$ 
for each $\alpha\in\supp(g_k)$. 

Rewriting~\eqref{eq:ett} as the sum
\be\label{eq:int rewrite}
M_{f}(s,t)
=\sum_{\alpha\in \supp(g_k)}\frac{h_\alpha(t) }{\<\mu_k,s\>-\<\nu_k,t\>}\int_{\RR_+^n}\frac{z^{s+\alpha}}{f(z)^{t+1}}\,\frac{dz}{z},
\ee
for some linear polynomials $h_\alpha(t)$, 
noting that each term of \eqref{eq:int rewrite} is a translation of the original Euler--Mellin integral. 
By Theorem~\ref{thm:expgro}, the term corresponding to $\alpha$ 
converges on the domain given by 
$\tau+1>0$ and 
\[
\<\mu_j,\sigma+\alpha\>>\<\nu_j,\tau+1\>,\quad \text{for } j=1,\dots, N, 
\]
where the latter is equivalent to
\[
\< \mu_j,\sigma\>>\<\nu_j,\tau+1\> - \< \mu_j,\alpha\> 
= \<\nu_j,\tau\>-d_j^\alpha, \quad \text{for } j=1,\dots,N.
\]
The sum~\eqref{eq:int rewrite} converges on the 
intersection of these domains, which is given by 
\begin{align*}
\tau +1 &> 0,  \\
   \< \mu_j,\sigma\> &> \<\nu_j,\tau\>  \quad \text{if } j \neq k,  \text{ and}\\
   \< \mu_k,\sigma\> &> \<\nu_k,\tau\> - d_k, \end{align*}
where $d_k := \min \{ d_k^\alpha \mid \alpha\in\supp(g_k)\}$.
Since $d_k$ is by definition strictly greater than 0,~\eqref{eq:int rewrite} 
has a strictly larger domain of convergence than~\eqref{eq:eulermellin}; 
we say that it has been extended by the ``distance" $d_k$ 
in the direction determined by $\mu_k$. 

Before iterating this procedure, we set some notation. 
Let $G_k$ be the semigroup generated by the integers 
$\{d_k^\alpha\}\subseteq \NN$. 
Let $\beta = (\alpha_1,\dots,\alpha_q)$ be an ordered $q$-tuple with 
$\alpha_i \in \supp(f)$ for each $i$. 
We sometimes write $\beta$ as an exponent of $z$, viewing $\beta = \alpha_1 + \dots + \alpha_q$. 
Similarly, set $d_k^\beta := d_k^{\alpha_1} + \dots + d_k^{\alpha_q} \in G_k$.

Now after $q$ iterations, let $\mu_{j(i)}$ denote the direction of the extension 
in the $i$th iteration. Let 
$d_{j(i)}^{\beta_i} := d_{j(i)}^{\alpha_1} + \dots + d_{j(i)}^{\alpha_{i-1}}\in G_{j(i)}$ 
be the sum of the distances of the first $i-1$ components of $\beta$ in 
the direction $\mu_{j(i)}$. 
Then there is a rational function of the type
\be\label{eq:first hs}
L_{\beta}(s,t) = \prod_{i=1}^q\frac{h_{\beta_i}(t)}{\<\mu_{j(i)},s\>-\<\nu_{j(i)},t\> + d_{j(i)}^{\beta_i}},  
\ee
where $h_\beta (t) := (h_{\beta_1}(t), \dots, h_{\beta_q}(t))$ 
is an ordered $q$-tuple of linear polynomials 
such that $M_f$ can be expressed 
as a finite sum of translations of the original Euler--Mellin integral: 
\be\label{eq:induction}
M_f(s,t) = \sum_{\beta} L_{\beta}(s,t) \int_{\RR_+^n}\frac{z^{s+\beta}}{f(z)^{t+q}}\,\frac{dz}{z}.
\ee
Fixing $k$, we next expand the domain of convergence of~\eqref{eq:induction} 
in the direction determined by $\mu_k$. 
This is achieved through simultaneous expansion of the 
domains of convergence of all terms, arguing as above. 
This yields the expression
\begin{align}\label{eq:result int}
M_f(s,t) 
&= \sum_{\beta} L_{\beta}(s,t) \sum_{\alpha\in \supp(g_k)}
\frac{h_{(\beta,\alpha)_{q+1}}(t) }{\<\mu_k,s\>-\<\nu_k,t\>+d_k^\beta}
\int_{\RR_+^n}\frac{z^{s+\beta+\alpha}}{f(z)^{t+q+1}}\,\frac{dz}{z}\\
&= \sum_{\beta'}L_{\beta'}(s,t) \int_{\RR_+^n}\frac{z^{s+\beta'}}{f(z)^{t+q'}}\,\frac{dz}{z}, \nonumber
\end{align}
where $\beta' = (\beta,\alpha)$, $q' = q+1$, 
and the resulting rational function 
$L_{\beta'}(s,t)$ is given by 
\[
L_{\beta'}(s,t) = L_{\beta}(s,t)\, \frac{h_{\beta'_{q'}}(t) }{\<\mu_k,s\>-\<\nu_k,t\>+d_k^\beta}.
\]

Since the convergence domain of each term in~\eqref{eq:induction} 
is extended by the distance $d_k$ in the direction determined by $\mu_k$, 
the convergence domain of the sum is similarly extended. 
In addition, since $d_k^\alpha >0$, 
we have that $d_k^{\beta+\alpha} > d_k^{\beta}$; 
therefore, the products $L_\beta(s,t)$ will never repeat factors 
in their denominators. As~\eqref{eq:result int} is in the same form 
as~\eqref{eq:induction}, we may iterate this procedure to extend the domain of convergence.

Finally, note that after $q$ iterations that have extended the domain of convergence of $M_{f}(s,t)$ 
in the direction determined by $\mu_j$ for $q_j$ of the $q$ steps, we obtain a 
meromorphic function on the tube domain given by $(s,t)\in\CC^{n+m}$ such that $\tau + \sum_{j=1}^N q_j = \tau + q >0$ and 
\[
\<\mu_j,\sigma\> > \<\nu_j,\tau\>-q_j d_j, \quad \text{for } j=1,\dots, N. 
\]
Continuing, $M_{f}(s,t)$ can be extended to a meromorphic function 
on $\CC^{n+m}$ as in~\eqref{eq:mero}. 
We note that because the denominator of the products of the rational functions 
$L_{\beta}(s,t)$ never has repeated terms, 
all poles of the extended Euler--Mellin integral are simple. 
Therefore by the removable singularities theorem, $\Phi_f(s,t)$ in~\eqref{eq:mero} 
is an entire function, as desired. 
\end{proof}

The entire function $\Phi_f(s,t)$ is of great interest to the study of $A$-hypergeometric functions. 
The Gamma functions appearing in~\eqref{eq:mero}
might have introduced some unnecessary zeros in the meromorphic continuation of the Euler--Mellin integral, which hinder $A$-hypergeometric applications. 

\begin{remark}\label{rm:poleSemiGroups}
In the proof of Theorem~\ref{thm:mel}, we see that the linear form 
$\< \mu_k,\sigma\> - \<\nu_k,\tau\> - d$ 
appears in the denominator of some rational function $L_\beta$ 
if and only if $d\in G_k$. 
Hence if $G_k \neq \NN$, then our meromorphic continuation has introduced unnecessary zeros 
into the entire function $\Phi_f(s,t)$. 
\end{remark}

\begin{remark}\label{rm:onePolyZeros}
If $m=1$, then $h_{\beta_i}(t) = k_{\beta_i} (t+i)$ for some constant 
$k_{\beta_i}$, where $h_{\beta_i}$ is as in~\eqref{eq:first hs}. 
Therefore each $L_\beta$ is divisible by $(t)_{i+1} = t(t+1)\cdots (t+i)$, 
which can thus be factored outside of the sum~\eqref{eq:induction}. 
In particular, $\widetilde{\Phi_f}(s,t) := \Gamma(t)\Phi_f(s,t)$ is an entire function. 
\end{remark}

We conclude this section with examples to illustrate 
Theorem~\ref{thm:mel} and our recent remarks. 

\begin{example}
Consider the case of $m+1$ linear functions of one variable,
\be\label{eq:ex Mf}
M_f(s,t) = \int_0^\infty \frac{z^s}{(1 + z)^{t_0}(c_1+z)^{t_1}\cdots (c_m + z)^{t_m}} \frac{dz}{z}.
\ee
Note that we have reindexed $t$ for this example. 
When $m=0$,~\eqref{eq:ex Mf} is the Beta function. 
Here $\Phi_f(s,t) = 1/\Gamma(t)$,
or with the notation of Remark~\ref{rm:onePolyZeros}, $\widetilde{\Phi_f}(s,t) = 1$. 
When $m = 1$, we showed in Example~\ref{ex:2F1classical} that
\[
 \Phi_f(s,t) = \frac 1 {\Gamma(t_0+t_1)}\, {}_2 F_1(t_1,t_0+t_1-s;t_0+t_1;1-c_1).
\]
This equality is obtained by the change of variables $w = z/(1+z)$ and 
application of the generalized binomial theorem. 
By similar calculations for $m = 2$,
\[
\Phi_f(s,t) = \frac{1}{\Gamma(t_0+t_1+t_2)}\, 
 	F_1(t_0+t_1+t_2-s, t_1,t_2;t_0+t_1+t_2; 1-c_1,1-c_2),
\]
where $F_1$ denotes the first Appell series. 
For arbitrary $m$ and $|c_i|<1$, 
\[
\Phi_f(s,t) = \frac 1 {\Gamma(t_0+ |t|)} \, \sum_{k\in \NN^{m}} 
\frac{(t_0+|t|-s)_{|k|}}{(t_0+|t|)_{|k|}}\, \frac{(t)_{k}}{k!}\, (1-c)^{k},
\]
where $t = (t_1,\dots,t_m)$, $|t| = t_1 + \dots + t_m$,
and $(t)_k = (t_1)_{k_1}\cdots (t_m)_{k_m}$.
\end{example}

\begin{example}\label{ex:MellinLinearPoly}
Finally, consider the case of one linear function of $n$ variables,
\[
M_f(s,t) = \int_{\RR_+^n} 
\frac{z_1^{s_1}\cdots z_n^{s_n}}{(1 + z_1 + \dots + z_n)^{t}} 
\frac{dz_1\wedge \cdots \wedge dz_n}{z_1\cdots z_n}.
\]
We claim that
\[
 M_f(s,t) = \frac{\Gamma(s_1)\cdots \Gamma(s_n)\Gamma(t-s_1-\dots-s_n)}{\Gamma(t)},
\]
and hence $\Phi_f(s,t) = 1/\Gamma(t)$.
This is clear when $n=1$ because we again have the Beta function. 
For $n>1$ one can argue by induction, making the change of variables 
given by $w_n = z_n$ and $w_i = z_i/(1+z_n)$ for $i\neq n$.
To generalize this example to an arbitrary simplex, consider the Euler--Mellin integral
\[
M_f(s,t) = \int_{\RR_+^n} 
\frac{z_1^{s_1}\cdots z_n^{s_n}}{(1 + z^{T_1} + \dots + z^{T_n})^{t}} 
\frac{dz_1\wedge \cdots \wedge dz_n}{z_1\cdots z_n},
\]
where the exponent vectors $T_i$ are the columns of an invertible matrix $T$. By the change of variables $w_i= z^{T_i}$ we find that
\[
 M_f(s,t) = \frac{\Gamma((T^{-1}s)_1)\cdots \Gamma((T^{-1}s)_n)\Gamma(t-|T^{-1}s|)}{|\det(T)|\,\Gamma(t)}.
\]
\end{example}

\section{Relation to coamoebas}
\label{sec:coamoebas}

For Theorems~\ref{thm:expgro} and~\ref{thm:mel} to hold, 
each $f_i(z)$ must be completely nonvanishing on the positive orthant. 
This is a strong restriction that many polynomials will not fulfill. 
However, the goal of this section is to modify this hypothesis by 
considering the coamoeba of $f(z)$.

The amoeba $\cA_f$ and the coamoeba $\cA'_f$ of a polynomial $f$ 
are defined to be the images of the zero set 
$Z_f=\{ z\in(\CC_*)^n \mid f(z)=0\}$
under the real and imaginary parts of the coordinate-wise 
complex logarithm mapping, $\Log$ and $\Arg$, respectively. 
More precisely, if $\Log(z)=(\log|z_1|,\ldots,\log|z_n|)$ 
and $\Arg(z)=(\arg(z_1),\ldots,\arg(z_n))$, then the amoeba and coamoeba of $f$ are, respectively, 
\begin{equation}\label{eq:def coamoeba}
\cA_f:=\Log(Z_f)\qquad \text{and}  \qquad \cA'_f:=\Arg(Z_f). 
\end{equation}

The amoeba $\cA_f$ lies in $\RR^n$; 
however, since the argument mapping is multivalued, 
the coamoeba $\cA'_f$ can be viewed either in the 
$n$-dimensional torus $\TT^n = (\RR/2\pi\ZZ)^n$ 
or as a multiply periodic subset of $\RR^n$. 
Amoebas were introduced by Gelfand, Kapranov, and Zelevinsky in~\cite{GKZa}, 
while the term coamoeba was first used by the third author in 2004 
at a conference at Johns Hopkins University. 

The following results provide the connection between coamoebas and 
Euler--Mellin integrals.

\begin{pro}\label{prop:coamoeba}
For $\theta \in \TT^n$, a polynomial $f(z)$ is completely nonvanishing 
on the set $\Arg^{-1}(\theta)$ if and only if $\theta \notin \overline{\cA'_f}$.
\end{pro}
\begin{proof}
The claim is equivalent to the statement 
\[
\overline{\cA'_f} = \bigcup_{\Gamma}\cA'_{f_\Gamma},
\]
where $f_\Gamma$ is the truncated polynomial with support $\Gamma$. 
This has been proven by Johansson \cite{J} 
and independently by Nisse and Sottile \cite{NS}.
\end{proof}

Thus for polynomials $f_1,\dots,f_m$ such that 
the closure of the coamoeba of $f(z) = \prod_{i=1}^m f_i(z)$ is a proper subset of 
$\TT^n$, there is a $\theta \notin \overline{\cA'_f}$ for which the 
Euler--Mellin integral with respect to $\theta$ is well-defined: 
\be\label{eq:theta eulermellin}
M_f^{\theta}(s,t) := \int_{\Arg^{-1}(\theta)} \frac{z^s}{f(z)^t} \frac{dz}{z}.
\ee
Note that after fixing the matrix $A$, and hence the set of monomials of the polynomials $f_i$, any choice of coefficients for the $f_i$ with positive real part ensures that $0\in\TT^n$. In particular, this means there is always a choice of coefficients for the $f_i$ so that $\overline{\cA'_f}$ is a proper subset of $\TT^n$. 
For another discussion on the existence of components of the complement of $\overline{\cA'_f}$, see Remark~\ref{rem:lopsided}. 

As \eqref{eq:theta eulermellin} differs from our earlier definition of the Euler--Mellin integral 
in~\eqref{eq:eulermellin} only by a change of variables, 
it is immediate that $\theta$-analogues of 
Theorems~\ref{thm:expgro} and~\ref{thm:mel} hold. 
In addition, a slight perturbation of $\theta$ does not impact 
the value of~\eqref{eq:theta eulermellin}. 

\begin{thm}
The Euler--Mellin integral $M_f^\theta$ of \eqref{eq:theta eulermellin} is 
a locally constant function in $\theta$. Thus it depends only on the choice 
of connected component $\Theta$ of the complement of $\overline{\cA'_f}$, 
and we thus write $M_f^\Theta := M_f^\theta$.  Accordingly, there is a meromorphic continuation of $M_f^\Theta$, denoted 
$\Phi_f^\Theta := \Phi_f^\theta$. 
\end{thm}
\begin{proof}
First consider the case $n=1$, and suppose that $\theta_1$ 
and $\theta_2$ lie in the same connected component of the complement 
of $\overline{\cA'_f}$; in fact, assume that the interval 
$[\theta_1,\theta_2]\subseteq \TT^n\setminus\overline{\cA'_f}$. 
In other words, $f(z)$ has no zeros with arguments in this interval, 
and hence $z^{s-1}/f(z)^t$ is analytic in the corresponding domain. 
Connecting the two rays $\Arg^{-1}(\theta_1)$ and $\Arg^{-1}(\theta_2)$ 
with the circle section of radius $r$ yields a closed curve, 
and the integral of $z^{s-1}/f(z)^t$ over this (oriented) curve is zero 
by residue calculus. By the proof of Theorem~\ref{thm:expgro}, 
the integral over the circle section tends to 0 as $r\rightarrow \infty$, 
so the two Euler--Mellin integrals $M_f^{\theta_1}$ and $M_f^{\theta_2}$ 
are equal.

In arbitrary dimension, we obtain the desired equality by considering 
one variable at a time while the remaining variables are fixed.
\end{proof}

\begin{example}\label{ex:2F1 first}
Revisiting the polynomial 
$f(z_1,z_2) = c_1 + c_2z_1 + c_3z_2 + c_4 z_1 z_2$ 
from the introduction, 
we see that if we choose $\theta = (\arg(c_1/c_2),\arg(c_1/c_3))$, then
\medskip
\[
\Phi_f^{\Theta}(s_1,s_2,t) =  
\frac{c_1^{s_1+s_2-t}c_2^{-s_1}c_3^{-s_2}}{\Gamma(t)^2}\,{}_2F_1\Big(s_1,s_2;t;1-\frac{c_1 c_4}{c_2c_3}\Big),
\medskip
\]
where $\Theta$ is the component of the complement of $\overline{\cA'_f}$ 
containing $\theta$. In accordance with Remark~\ref{rm:onePolyZeros}, 
we may ignore one of the factors $\Gamma(t)$ in the denominator, 
while ${}_2F_1/\Gamma(t)$ is the regularized Gauss hypergeometric function.
\end{example}

\section{Integral representations of $A$-hypergeometric functions}
\label{sec:gkz}

We now fix a connected component $\Theta$ of the complement of $\overline{\cA'_f}$ and 
study the entire function $\Phi_f = \Phi_f^\Theta$ from \eqref{eq:theta eulermellin}. 
In particular, we consider its dependence 
on the coefficients $c_i=\{c_{i,\alpha}\}$ of the polynomials $f_i$, 
where $f(z)=\prod_{i=1}^m f_i$ and $f_i = \sum_{j=1}^{r_i} c_{ij} z^{\alpha_{ij}}$ (so $r_i = |\supp(f_i)|$). 
In order to emphasize this dependence, we write $\Phi_f(s,t,c)$ 
rather than $\Phi_f(s,t)$. 
Generalizing \cite[\S6]{NP}, 
we show that $\Phi_f$ is an $A$-hypergeometric function 
in the sense of Gelfand, Kapranov, and Zelevinsky. 
More precisely, $c\mapsto\Phi_f(s,t,c)$ satisfies the 
$A$-hypergeometric system of partial differential equations, where 
the exponents $\alpha_{ij}$ of the $f_i$ 
provide a matrix $A$ via the Cayley trick, 
\be\label{eq:A}
A = \left[\begin{matrix}
	1&\cdots&1&	0&\cdots&0&	\quad& 0&\cdots&0\\ 
	0&\cdots&0&	1&\cdots&1&	\quad& 0&\cdots&0\\
	\vdots&\cdots&\vdots&	\vdots&\cdots&\vdots& \cdots& \vdots&\cdots&\vdots\\
	0&\cdots&0& 0&\cdots&0&\quad& 1&\cdots&1\\
	\alpha_{11}&\cdots&\alpha_{1r_1} & 
		\alpha_{21}&\cdots&\alpha_{2r_2}  & 
		\quad & \alpha_{m_1}&\cdots&\alpha_{mr_m}
\end{matrix}\right] \in \ZZ^{(m+n)\times (r_1+\cdots +r_m)}, 
\ee
and the desired homogeneity parameter is $\beta = -(t,s)$.
Set $r:= \sum_{i=1}^m r_i$. 

We first recall the definition of an $A$-hypergeometric system. 
For a vector $v\in\ZZ^r$, denote by $u_+$ and $u_-$ the unique vectors in 
$\NN^r$ with disjoint support  such that $u=u_+ - u_-$.

\begin{dfn}\label{def:gkz}
Let $A = (a_{ij})\in \ZZ^{(m+n)\times r}$ be a matrix. 
Define the differential operators $\square_u$ and $E_i$ to be 
\[
\square_u:=\Bigl(\frac{\partial}{\partial c}\Bigr)^{u_+}-\Bigl(\frac{\partial}{\partial c}\Bigr)^{u_-}
\quad\text{and}\quad 
E_i:=\sum_{j=1}^r a_{ij} \frac{\partial}{\partial c_j}.
\]
The \emph{$A$-hypergeometric system} $H_A(\beta)$ 
at $\beta\in\mathbb{C}^{m+n}$ is given by
\begin{align*}
\square_u F(c)&=0 \quad \text{for $u\in\ZZ^r$ with $Au=0$},\\
\text{and}\quad(E_i -\beta_i) F(c) &= 0 \quad \text{for $1\le i \le m+n$}. 
\end{align*}
A local multivalued analytic function $F$ that solves this system is called an 
\emph{$A$-hypergeometric function} with homogeneity parameter $\beta$.
\end{dfn}

The ideal $I_A$ cuts out an affine variety $X_A\subseteq \CC^r$, which has an action of an algebraic torus, $(\CC_*)^{m+n}$. 
To understand the role of the Euler operators $E_i-\beta_i$, note that a germ of a analytic function at a nonsingular point $c\in\CC^r$ that is 
annihilated by $c_1\frac{\partial}{\partial c_1}+c_2\frac{\partial}{\partial c_2}+\cdots+c_r\frac{\partial}{\partial c_r} - \beta_0$ is homogeneous, in the usual sense, of degree $\beta_0$. 
In general, the Euler operators in $M_A(\beta)$ force solutions to have weighted homogeneities. 
From this point of view, it becomes natural to fix $A$ and view $\beta$ as a parameter of $M_A(\beta)$. 

A key problem in the study of an $A$-hypergeometric system is to describe the variation with the parameter $\beta$ of its solution space of germs of analytic functions at a nonsingular point. To begin this study, one must first find solutions of $H_A(\beta)$, that vary nicely with $\beta$. 
Saito, Sturmfels, and Takayama presented an algorithm to compute such a basis for arbitrary $\beta$, called \emph{canonical series solutions}~\cite{SST}. 
However, because this algorithm uses Gr\"obner degeneration, the solutions it produces are not well-suited to the variation of $\beta$. 
For nonresonant $\beta$, Gelfand, Kapranov, and Zelevinsky computed a basis of Euler-type integral solutions~\cite{GKZc}. 
These are also unsuitable for understanding parametric behavior, as their domains of integration are not explicit and the family integrals formed there are not guaranteed to be linearly enough at nongeneric $\beta$. 
In contrast, since our meromorphic continuations of Euler--Mellin integrals are entire in $\beta$, they provide a new tool for describing the parametric variation of $A$-hypergeometric solutions. 

To this end, we now consider the behavior of the entire function $(s,t)\mapsto\Phi_f(s,t,c)$, as described in Theorem~\ref{thm:mel}, when $c$ is viewed as a variable. Let $\Sigma_A\subseteq\CC^r$ denote the singular locus of all $A$-hypergeometric functions, which is the hypersurface defined by the principal $A$-determinant (also known as the full $A$-discriminant)~\cite{GKZa}.

\begin{thm}\label{thm:gkz}
Let $c\in\CC^r\setminus\Sigma_A$ and 
let $\Theta$ be a connected component of $\RR^n\setminus\overline{\cA'_f}$, 
where $f$ is the polynomial
$f(z)=\prod_{i=1}^m f_i$ and $f_i = \sum_{j=1}^{r_i} c_{ij} z^{\alpha_{ij}}$. 
Then the analytic germ $\Phi_f^\Theta(s,t,c)$, for any $\theta\in\Theta$, 
has a (multivalued) analytic continuation to 
$
\mathbb{C}^{m+n}
\times
(\mathbb{C}^r\setminus\Sigma_A)
$ 
that is everywhere $A$-hyper\-geometric (in the variables $c$) with 
homogeneity parameter $\beta = -(t,s)$.
\end{thm}
\begin{proof} 
Let us first consider the case $\tau:=\Re t > 0$ and $\sigma:=\Re s\in\inter (\tau\Delta_f)$, 
where we have the integral representation
\be\label{eq:phi hyp}
\Phi_f^\Theta(s,t,c)=
\frac{1}{\prod_k\Gamma(\<\mu_k,s\>-\<\nu_k,t\>)}
\int_{\text{\rm Arg}^{-1}(\theta)}\frac{z^s}{f(z)^t}\frac{dz}{z}.
\ee

Fix a representative $\theta\in\Theta$. 
As $\theta$ is disjoint from $\overline{\cA'_f}$ for polynomials $f$ with 
coefficients $c$ near the original ones, say in a small ball $B(c)$, 
the integral in \eqref{eq:phi hyp} does indeed define an analytic germ $\Phi_f = \Phi_f^\theta(s,t,c)$. 
By Theorem~\ref{thm:mel}, $\Phi_f$ can be extended to an entire function 
with respect to the variables $s$ and $t$. 
In other words, $\Phi_f$ has been analytically extended to the infinite cylinder $\CC^{m+n} 
\times B(c)$. 

To see that $\Phi_f$ is an $A$-hypergeometric function with homogeneity parameter $\beta$ as given, 
we fix $s$ and $t$ under the above condition, noting that the product of Gamma functions in $\Phi_f$ 
is simply a nonzero constant. 
Thus it is enough to show that the integral itself is $A$-hypergeometric at $\beta$. 
This is accomplished through the argument of \cite[Thm.~5.4.2]{SST}, which applies since 
differentiation and integration 
may be interchanged 
because Euler--Mellin integrals are uniformly convergent by the bound in~\eqref{eq:bound}. See also \cite[Remark~2.8(b)]{GKZc}.

Having established that $\Phi_f$ is an $A$-hypergeometric function 
in the product domain given by $(s,t,c)$ in 
$(\RR_+\inter(\tau\Delta_f)+i\RR^n)\times(\RR_+^m\times i\RR^m)
\times B(c)$, 
it follows from 
the uniqueness of analytic continuation that its extension to the cylinder 
$\CC^{m+n}\times B(c)$ will remain $A$-hypergeometric. 
Now for each fixed $(s,t)$, there is a (typically multivalued) 
analytic continuation of $c\mapsto\Phi_f = \Phi_f^\theta(s,t,c)$ 
from $B(c)$ to all of $\CC^r\setminus\Sigma_A$.  
As these continuations still depend analytically on $s$ and $t$, 
we have now achieved the desired analytic continuation to the full product 
domain $\CC^{m+n} \times (\CC^r\setminus\Sigma_A)$. 
The uniqueness of analytic continuation again guarantees that 
$\Phi_f$ will everywhere satisfy the $A$-hypergeometric system with 
the homogeneity parameter $\beta$, as desired.
\end{proof}

In \cite{GKZc, SST}, for generic (nonresonant) parameters $\beta$, a basis of solutions of $H_A(\beta)$ are given via generalized Euler integrals. 
The set of nongeneric, or resonant, parameters is given by the union over all faces $\Gamma$ of the polyhedral cone $\RR_{\geq 0}A$ over the columns of $A$ of $\ZZ^{m+n}+\CC \Gamma$.
The dimension of the solution space in this case is given by $\vol(A)$, which is $(m+n)!$ times the Euclidean volume of the convex hull of $A$ and the origin. 
Under certain conditions on $A$, a similar result holds for our meromorphic continuations of Euler--Mellin integrals. 
We consider here the case that $f$ is a univariate polynomial, followed by another example for which linear independence of these integrals can be seen directly by analyzing coamoebas. 
Such linear independence will also be addressed later in Corollary~\ref{cor:LopsidedIntegralIndependence}. 

\begin{pro}\label{prop:indep}
If $\beta$ is nonresonant and $n=1$, then 
the extended Euler--Mellin integrals 
$\Phi_f^\Theta(s,t,c)$, where $\Theta$ ranges over the components of  
$\TT^{n}\setminus\overline{\cA'_f}$, 
provide a basis of solutions to $H_A(\beta)$.  \end{pro}
\begin{proof}
For a generic choice of coefficients $c$, $f$ will have distinct roots 
$r_1,r_2,\dots,r_{\vol(A)}$ with distinct arguments 
$0\leq \theta_1<\theta_2<\dots<\theta_{\vol(A)}<2\pi$. 
That is, $\overline{\cA'_{f}}$ will have $\vol(A)$-many components in its complement in $\TT^1$.
For convenience, we use the convention that $\theta_{\vol(A)+1}:= \theta_1$.
Now fix $0<\epsilon\ll 1$ so that the circles 
$B_\epsilon(r_i) := \{ z\in\CC \mid  |z-r_i|=\epsilon \}$, 
viewed as 1-chains oriented counterclockwise, have disjoint supports for $1\leq i\leq \vol(A)$. 
By \cite{GKZc}, the integrals 
\[
\int_{ B_\epsilon(r_i) } \frac{z^s}{f(z)^t} \frac{dz}{z} 
\qquad \text{for }\ 1\leq i\leq \vol(A) 
\]
are linearly independent, forming a basis for the solution space of $H_A(\beta)$. 
By the convergence of $M_f^{\theta_i}(s,t,c)$ from Theorems~\ref{thm:expgro} and~\ref{thm:mel}, we see that as non-compact chains, 
$\Arg^{-1}(\theta_i)-\Arg^{-1}(\theta_{i+1})$ and 
$B_\epsilon(r_i)$ are homologous, providing the second statement. 
For a nongeneric choice of coefficients $c$ for $f$, which still lie away from $\Sigma_A$, a similar argument implies linear independence of the Euler--Mellin integrals given by the distinct components of the complement of $\overline{\cA'_f}$. 
\end{proof}

\begin{example}\label{ex:gauss with $t$}
In Example~\ref{ex:2F1 first}, it was shown that if 
$f(z) = c_1+c_2 z_1 + c_3 z_2 + c_4 z_1 z_2$ and 
$\theta$ is near $(\arg(c_1/c_2),\arg(c_1/c_3))$, then 
\[
\Phi_f^\theta(s,t,c) = \frac{c_1^{s_1+s_2-1}c_2^{-s_1}c_3^{-s_2}}{\Gamma(t)^2}
	\ {}_2F_1\hspace{-.8mm}\left(\hspace{-1mm}s_1,s_2;t;1-\frac{c_1c_4}{c_2c_3}\right). 
\]
By Theorem~\ref{thm:gkz}, this is a solution of $H_A(\beta)$ for 
\[
A = \left[\begin{matrix}
1&1&1&1\\0&1&0&1\\0&0&1&1
\end{matrix}\right]
\quad\text{and}\quad \beta = -(t,s).
\] 
Now consider points $c$ of the form $(1,i,i,c_4)$, where $c_4$ is near 1. 
Define the polynomials 
\begin{align*}
f_\rho &:= 1+e^{\frac{\pi i}{2}\rho}z_1 + e^{\frac{\pi i}{2}\rho}z_2 + c_4 z_1 z_2  
	&\text{for } \ 0\le \rho\le 1\phantom{.}\\
\text{and}\quad 
g_\rho &:= 
	1+e^{\frac{\pi i}{2}(2-\rho)}z_1 + e^{\frac{\pi i}{2}(2-\rho)}z_2 + c_4 z_1 z_2 
	&\text{for } \ 0\le \rho\le 1. 
\end{align*}
As shown in Figure~\ref{fig:gauss coamoebas}, the coamoeba for $f_1 = g_1$ has two connected components: one containing $(0,0)$ and another containing $(\pi,\pi)$. 
These yield two solutions to $H_A(\beta)$ at $c = (1,i,i,c_4)$ by Theorem~\ref{thm:gkz}, namely, $\Phi_{f_1}^{(0,0)}( s,t,c)$ and $\Phi_{f_1}^{(\pi,\pi)}( s,t,c)$. 
In addition, $(0,0)\notin\overline{\cA'_{f_\rho}}$ and 
$(\pi,\pi)\notin\overline{\cA'_{g_\rho}}$ for all $\rho$, so we   
let $\Phi_{f_\rho}^{(0,0)}( s,t,c)$ and $\Phi_{g_\rho}^{(\pi,\pi)}( s,t,c)$ 
denote the entire functions corresponding 
to $f_\rho$ and $g_\rho$, respectively. 

\begin{figure}[t]
\begin{tabular}{ccc}
\setlength{\unitlength}{1mm}
\begin{tikzpicture}(40,40)
  \put(20,20){\circle*{.75}}
  \put(20.5,21){$(0,0)$}
  \put(0,0){\line(1,0){40}} 
  \put(0,0){\line(0,1){40}} 
  \put(40,0){\line(0,1){40}} 
  \put(0,40){\line(1,0){40}} 
\end{tikzpicture}
&
\hspace{5cm}
\setlength{\unitlength}{1mm}
\begin{tikzpicture}(40,40)
 \filldraw[fill=black] (1,0) -- (1,1) -- (3,1) -- (3,0)-- cycle;
 \filldraw[fill=black] (0,1) -- (1,1) -- (1,3) -- (0,3)-- cycle;
 \filldraw[fill=black] (3,1) -- (3,3) -- (4,3) -- (4,1)-- cycle;
 \filldraw[fill=black] (1,3) -- (3,3) -- (3,4) -- (1,4)-- cycle;
 \put(20,20){\circle*{.75}}
 \put(20.5,21){$(0,0)$}
 \put(39,39){\circle*{.75}}
 \put(30.25,34.75){$(\pi,\pi)$}
\end{tikzpicture} 
&
\hspace{1cm}
\setlength{\unitlength}{1mm}
\begin{tikzpicture}(40,40)
  \put(39,39){\circle*{.75}}
  \put(30.25,34.75){$(\pi,\pi)$}
  \put(0,20){\line(1,0){40}} 
  \put(20,0){\line(0,1){40}} 
\end{tikzpicture}
\end{tabular}
\hspace*{4cm}
 \caption{The coamoebas of the polynomials $f_0$, $f_1=g_1$, and $g_0$, 
 respectively, shown inside the fundamental domain $[-\pi,\pi]\times[-\pi,\pi]$ 
 of $\TT^2$ in $\RR^2$.}
 \label{fig:gauss coamoebas}
\end{figure}
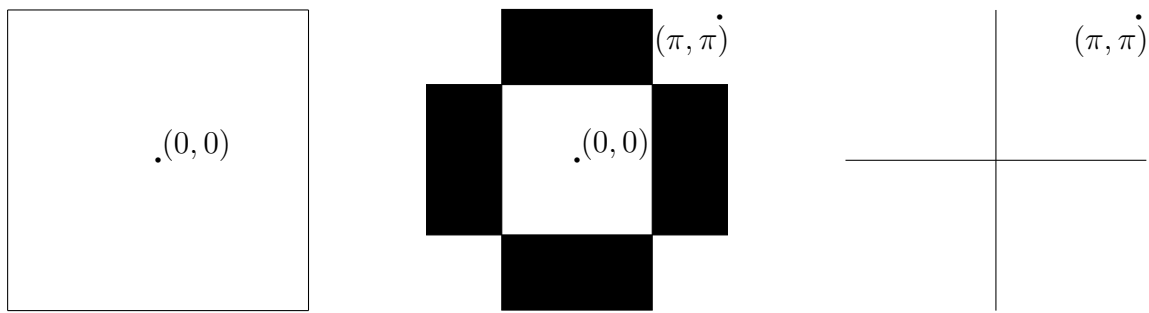

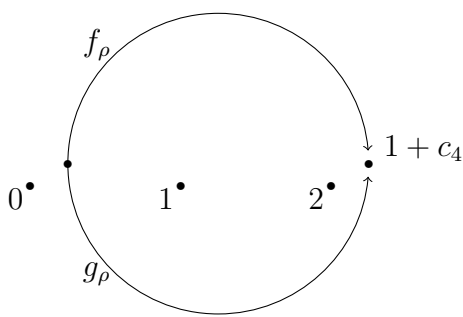
\begin{figure}[h]
\setlength{\unitlength}{1mm}
\begin{tikzpicture}(80,40)
  \draw [->] (2.5,2.3) arc (180:5:20mm);
  \draw [->] (2.5,2.3) arc (180:355:20mm);
  \put(25,23){\circle*{1}}
  \put(20,20){\circle*{1}}
  \put(17,17){$0$}
  \put(40,20){\circle*{1}}
  \put(37,17){$1$}
  \put(60,20){\circle*{1}}
  \put(57,17){$2$}
  \put(65,23){\circle*{1}}
  \put(67,24){{$1+c_4$}}
  \put(27,38){$f_{\rho}$}
  \put(27,08){$g_{\rho}$}
\end{tikzpicture}
\caption{The loop $L$ in the $(1+c_4)$-complex plane is given by the reverse of the arrow labeled $f_\rho$, followed by the arrow labeled $g_\rho$.}
\label{fig:gauss continuation}
\end{figure}

Let $L$ be the loop in the coefficient space given by first following the coefficients in the reverse of $f_\rho$ and then those in $g_\rho$. 
Since $f_1 = g_1$ and $\Phi_{g_0}^{(\pi,\pi)}( s,t,c) = e^{(s_1+s_2)\pi i} \Phi_{f_0}^{(0,0)}( s,t,c)$,  we can explicitly perform an analytic continuation of 
$\Phi_{f_1}^{(0,0)}( s,t,c)$ along the loop $L$, see Figure~\ref{fig:gauss continuation}. 
Therefore if the monodromy of $H_A(-t,-s)$ is irreducible, then 
$\Phi_{f_1}^{(0,0)}( s,t,c)$ and $\Phi_{f_1}^{(\pi,\pi)}( s,t,c)$ form a basis 
for the solution space of $H_A(-t,-s)$ of analytic germs at $(1,i,i,c_4)$. 
As the monodromy irreducibility of $H_A(-t,-s)$ is equivalent to the nonresonance of 
$\beta$~\cite{beukers-irred, sw-irred nonres}, the conclusion of Proposition~\ref{prop:indep} also holds in this case.
\end{example}

\section{Lopsided coamebas and Mellin--Barnes integrals}
\label{sec:MellinBarnes}

An alternative approach for constructing integral representations of $A$-hypergeometric functions is through  \emph{Mellin--Barnes integrals}~\cite{N, beukers-monodromy}. 
The main result of this section is Theorem~\ref{thm:EM-MB}, which identifies the set of Mellin--Barnes integrals of a system $H_A(\beta)$ with a certain subset of the set of Euler--Mellin integrals. 
This in turn implies Corollary~\ref{cor:LopsidedIntegralIndependence}, which asserts the linear independence of certain collections of meromorphic continuations of Euler--Mellin integrals at generic parameters $\beta$.  

The \emph{Gale dual} of $A$ is an integer $r\times (r-m-n)$-matrix $B$ with relatively prime maximal minors such that $AB = 0$. Typically we will not require that the condition on maximal minors holds; in this case, the matrix $B$ is called a \emph{dual matrix} of $A$. 
Given a dual matrix $B$ of $A$, let $\ZZ[B]$ denote the sublattice of $\ZZ^{r-m-n}$ generated by the rows of $B$. Also, consider the zonotope 
\[
\bZ_B := \left\{\frac{\pi}{2}\sum_{i=1}^r \mu_i b_i \ \Bigg|\ |\mu_i|<1\right\},
\]
where $b_i$ denotes the $i$th row of the matrix $B$. 
For connections between camoebas and Gale duals, see \cite{NP,lopsided}.

\begin{dfn}\label{def:MB}
Fix a Gale dual $B$ of $A$, and let $\gamma$ be such that $A\gamma = \beta$. Then for $c\in \CC^r$, the \emph{Mellin--Barnes integral} has the form
\begin{align}
\label{eqn:MB}
L(c) = L(c_1, \dots, c_r) = \int_{(i\RR)^m}\prod_{i=1}^{r} \Gamma(-\gamma_i - \<b_i, w\>)c_i^{\gamma_i + \<b_i,w\>}\,dw,
\end{align}
where $dw = dw_1\wedge \dots \wedge dw_m$. 
\end{dfn}

We say that $\beta\in\CC^{m+n}$ is \emph{totally nonresonant} for $A$ if the shifted lattice $\beta + \ZZ^{m+n}$ has empty intersection with any hyperplane spanned by any $m+n-1$ linearly independent columns of $A$.

Given $\theta\in\TT^n$ and $c\in(\CC_*)^r$, let 
\[
 L^{\theta}(c)
 := L(c_1 e^{i\<\alpha_1, \theta\>}, \dots, c_r e^{i\<\alpha_r, \theta\>}).
\]
The following Mellin--Barnes result summarizes Corollary~4.2, Theorem~3.1, and Proposition~4.3 of \cite{beukers-monodromy}. 

\begin{thm}[\cite{beukers-monodromy}] 
\label{thm:beukers}
Fix $\theta\in\TT^n$ and $c\in(\CC_*)^r$. 
\begin{enumerate}
\vspace{-1ex}
\item 
	If $\Arg(c_\theta)B \in \inter(\bZ_B)$, 
	then the integral $L^\theta(c)$ converges 
	absolutely. 
\item 
	If $\Arg(c_\theta)B \in \inter(\bZ_B)$ 
	and $\gamma_i<0$ for each $i$, 
	then $L^\theta(c)$ satisfies the 
	$A$-hypergeometric system $H_A(\beta)$  
	as a germ of an analytic function at  
	$c$.
\item \label{item3}
	Suppose that $\beta$ is 
	totally nonresonant for $A$ and 
	$\theta_1, \dots, \theta_k\in\TT^n$ are 
	such that the 
	$(r-m-n)$-tuples 
	$\Arg(c_{\theta_1})B, \dots, 
	\Arg(c_{\theta_k})B$ are distinct and 
	contained in $\inter(\bZ_B)$. Then 
	as germs at $c$, the 
	Mellin--Barnes integrals 
	$L^{\theta_1}(c)$, \dots, 
	$L^{\theta_k}(c)$ 
	are linearly independent. 
\end{enumerate}
\end{thm}

It follows from results in \cite{lopsided} that each 
$(r-m-n)$-tuple 
$\Arg(c_{\theta_i})B$ is contained in the translated lattice $\Arg(c)B + 2\pi \ZZ[B]$. Thus for each set of points $\inter(\bZ_B)\cap (\Arg(c)B + 2\pi \ZZ[B])$, there is an associated set of linearly independent solutions to $H_A(\beta)$, given by the corresponding Mellin--Barnes integrals.

To relate Mellin--Barnes integrals to Euler--Mellin integrals, we first recall the definition of the lopsided coamoeba of a polynomial $f(z)$. 
The polynomial
$F(c,z) := \sum_{\alpha\in A} c_\alpha z^\alpha$, 
with coefficients $c$ also viewed as variables,  has coamoeba $\cA'_F$ contained in $\TT^{n+r}$. 
Given $f(z)=\sum_{\alpha\in A} c_\alpha z^\alpha$ with fixed coefficients $c$, 
the \emph{lopsided coamoeba} of $f$, 
denoted by $\cD_f$, is the intersection of $\cA'_F$ with the sub-$\TT^n$-torus of $\TT^{n+r}$ obtained by fixing $\Arg(c)$ as prescribed by $f$. The lopsided coamoeba $\cD_f$ is viewed as a subset of $\TT^n$.

There are several ways to define the lopsided coamoeba; for a more exhaustive treatment, we refer the reader to~\cite{lopsided}. 
The name ``lopsided coamoeba" can be misleading; the lopsided coamoeba is not a actually a coamoeba itself, but it can be viewed as a crude approximation of the coamoeba. 
See Figure~\ref{fig:Lopsided} for a comparison. 
There is a natural inclusion $\cA'_f\subseteq \cD_f$, and in particular, each component of 
$\TT^n\setminus\overline{\cD_f}$
is contained in a component of 
$\TT^n\setminus\overline{\cA'_f}$.
The induced map on components of the closures of these coamoebas is injective but in general, is not surjective. 

\begin{figure}
\centering
\includegraphics[width=40mm]{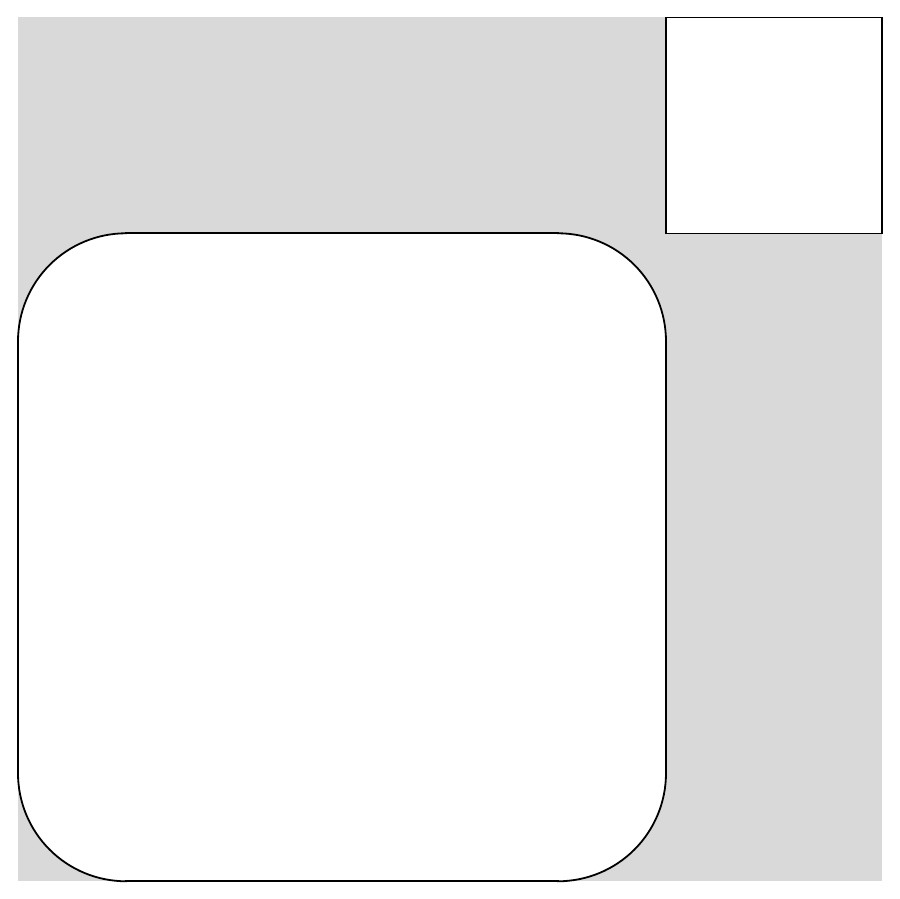}
\hspace{1cm}
\includegraphics[width=40mm]{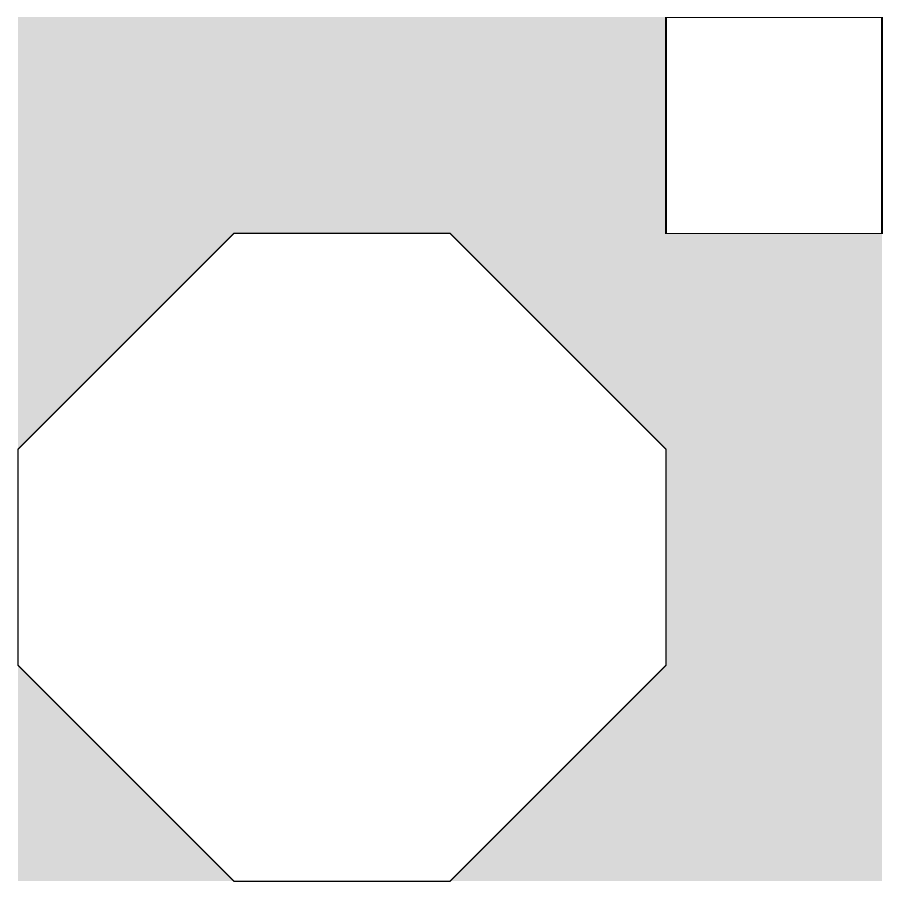}
\vspace{-4mm}
\caption{The coamoeba and the lopsided coamoeba of $f(z_1, z_2) = 1 + z_1 + z_2 + i z_1 z_2$.}
\label{fig:Lopsided}
\end{figure}

\begin{remark}\label{rem:lopsided}
The main result of~\cite{lopsided} states that the set of connected components of 
the complement of $\overline{\cD_f}$ 
is equipped with an \emph{order map}. 
Fix $\alpha=(\alpha_1, \dots, \alpha_m)\in \bigoplus_{j=1}^m \supp(f_j)$, and 
let $\arg_\pi$ denote the principal branch of the argument mapping.  
Denote by 
$\left[
	\arg_\pi\left(
		\frac{c_{\alpha_{jk}} 
		e^{i\<\alpha_{jk},\theta\>}}{
		c_{\alpha_j}  e^{i\<\alpha_j,\theta\>}}
	\right)
\right]$ the row vector of length $r$ with entries as shown, indexed by $(j,k)$ in lexicographic order. This indexing corresponds to the labeling of the columns of $A$, as shown in~\eqref{eq:A}. 
Define the function $v_\alpha\colon\RR^n\to\RR^{r-m-n}$ to be 
\[
v_\alpha(\theta) :=  \left[
	\arg_\pi\left(
		\frac{c_{\alpha_{jk}} 
		e^{i\<\alpha_{jk},\theta\>}}{
		c_{\alpha_j}  e^{i\<\alpha_j,\theta\>}}
	\right)
\right]B.
\]
The properties of this function is given in detail in \cite[Thm. 4.1 and Prop. 4.4-5]{lopsided} and in \cite[Thm. 2.3.10]{jensthesis}.
There it is shown that $v_\alpha$ is well-defined on $\TT^n$; and furthermore, it is constant on each component of $\TT^n\setminus\overline{\cD_f}$. In fact, the map from the set of connected components of $\TT^n\setminus\overline{\cD_f}$ to $\inter(\bZ_B)\cap (\Arg(c)B + 2\pi \ZZ[B])$ given by 
\[
\Theta \mapsto v_\alpha(\theta) \ \, \text{for some representative } \theta \in \Theta
\]
is surjective and independent of choice of $\alpha\in A$. If the maximal minors of $A$ are relatively prime, then this map is also injective. 
\end{remark}

We are now prepared to approach the first result of this section, which states explicitly how the order map for the set of components of the complement of the closure of the lopsided coamoeba lifts to a bijection between the set of Mellin--Barnes integrals arising from the points in the translated lattice $\Arg(c)B + 2\pi \ZZ[B]$ and the set of Euler--Mellin integrals arising from the components of the complement of the lopsided coamoeba of the polynomial with coefficients $c$.
To show this, we consider a dual matrix $B$ �of $A$ with a special form. 

\begin{convention}\label{conv:dual}
Without loss of generality, we may assume that $A$ is of the form
\[
A = \left[\begin{array}{ccc}
1 & 1 & 1 \\
 0 & A_\rI & A_\rII
 \end{array}\right],
\]
where $A_\rI$ is a nonsingular $n\times n$-matrix. Let $B$ denote the dual matrix of the form 
\[
B = \left[\begin{array}{c}
-a_0\\
A_\rI^{-1}A_\rII\\
-I_m
\end{array}\right]D,
\]
where $a_0$ is defined by the property that  each column sum of $B$ is zero and $D$ is an integer diagonal matrix chosen so that $B$ becomes an integer matrix. 
Let $g_A$ (respectively, $g_B$) denote the greatest common divisor of the maximal minors of the matrix $A$ (respectively, $B$).
\end{convention}

\begin{lem}\label{lem:EM-MB}
Following Convention~\ref{conv:dual}, there is an equality 
\[
\frac{g_B}{g_A} = \frac{|\det(D)|}{|\det(A_\rI)|}.
\]
\end{lem}
\begin{proof}
We may assume that $g_A = 1$. 
Following~\cite[Prop. 4.2]{N}, this equality holds precisely when $A$ can be extended to a $r\times r$ unimodular matrix 
\[
\widetilde A = \left[\begin{array}{ccc}
	1 & 1 & 1\\
	0 & A_\rI & A_\rII\\
	* & * & *
    \end{array}\right]
\ \text{with inverse} \ \, 
 {\widetilde A}^{-1} 
= \left[
	\begin{array}{cc} * & \widetilde B
	\end{array}
	\right]
= \left[
	\begin{array}{cc}
	* & \tilde b_0 \\
	* & {\widetilde B}_1\\
	* & {\widetilde B}_2
    \end{array}\right].
\]
It follows that $\widetilde B$ is a Gale dual of $A$, and by the Schur complement formula, $|A_\rI| = |{\widetilde B}_2|$. As $B = {\widetilde B} T$ for some affine transformation $T$, the desired equality follows.
\end{proof}

\begin{thm}\label{thm:EM-MB}
For each $\theta\in \TT^n\setminus\overline{\cD_f}$, there is an equality 
$g_B\,L^\theta(c)=2\pi i\,e^{-i\<s,\theta\>}\,\Gamma(t)\,g_A\,M_f^\theta (c)$.
\end{thm}
\begin{proof}
We will give the proof after rewriting the homogeneity parameter as
$\beta = (-t, -A_\rI s)$. 
In this case, the Euler--Mellin integral is 
\begin{align}
M_f^\theta(c) 
& = \int_{\Arg^{-1}(\theta)}
	\frac{z^{A_\rI s}}{\left( 
	c_0 + c_1z^{\alpha_1} 
	+ \dots + c_n z^{\alpha_n} + c_{1+n}	
	z^{\alpha_{1+n}} + \dots + c_{m+n}
	z^{\alpha_{m+n}} \right)^t} 
	\, \frac{dz}{z}
	\nonumber \\
& = \frac{c_0^{|s|-t}}{{c_\rI}^s} 
	\int_{\Arg^{-1}(\tilde\theta)}
	\frac{z^{A_\rI s}}{\left(
	1 + {z}^{\alpha_1} 
	+ \dots + {z}^{\alpha_n} + 
	x_{1}^{1/{d_1}}
	{z}^{\alpha_{1+n}} + \dots + 
	x_{m}^{1/{d_m}}
	{z}^{\alpha_{m+n}} \right)^t} 
	\, \frac{d{z}}{{z}},
	\label{eq:Mtransf}
\end{align}
where $x_i = c^{B_{i}}$ and $B_i$ denotes the $i$th column of $B$. 
Denote by $M_f^\theta(x)$ the function given by the integral in~\eqref{eq:Mtransf}, and let $\phi = \Arg(x)$. 
Note that $\theta\in \TT^n\setminus\overline{\cD_f}$ is equivalent to 
$\theta \in \TT^{m+n}\setminus\overline{\cA'_F}$, which in turn is equivalent to 
the convergence of the integral 
\begin{align*}
\int_{\Arg^{-1}(\tilde \theta, \phi)}\frac{z^{A_\rI s} x^w}{\left(
	1 +  z^{\alpha_1} + \dots + z^{\alpha_n} 
	+ x_{1}^{1/{d_1}}z^{\alpha_{1+n}} 
	+ \dots + x_{m}^{1/{d_m}}z^{\alpha_{m+n}}
	\right)^t} \, \frac{dz\wedge dx}{z\,x}.
\end{align*}
However, this integral is precisely the Mellin transform with respect to $x$ of $M_f^\theta(x)$. Consequently,  
\begin{align*}
\{\cM\,M_f^{\theta}(x)(w)\}& = \frac{|\det(D)|}{|\det(A_\rI)|}\int_{\Arg^{-1}(\tilde \theta, \tilde \phi)}\frac{z^{s-A_\rI^{-1}A_2Dw} v^{Dw}}{(1 + z_1 + \dots + z_n + x_{1} + \dots + x_{m})^t} \, \frac{dz\wedge dx}{z\,x}\\
& = \frac{|\det(D)|}{|\det(A_\rI)|}\frac{\Gamma(s-A_\rI^{-1}A_\rII Dw)\Gamma(Dw)\Gamma(t-|Dw|-|s|+|A_\rI^{-1}A_\rII Dw|)}{\Gamma(t)},
\end{align*}
see Example~\ref{ex:MellinLinearPoly}. Turning to the Mellin--Barnes integral, fix parameters $\gamma_\rI = -s -A_\rI^{-1}A_\rII \gamma_\rII$ and $\gamma_0 = |s|-t+\<b_0, \gamma_\rII\>$. Assuming we have that $s_j>0$ for all $j$, that $t>|s|$, and that $-1 \gg\gamma_\rII>0$, it follows that $\gamma_k < 0$ for all $k$. Furthermore, with $a_i$ denoting the $(i+1)$-st column of $A$, 
\[
\sum_{i=0}^{r-1} \gamma_i a_i
= A\gamma = \left[\begin{array}{c}-t\\-A_\rI s\end{array}\right].
\]
For ease of display, set 
$\Psi := t-|Dw|-|s|+|A_\rI^{-1}A_\rII Dw|$. 
Then 
\begin{align*}
L^\theta(c) 
& = 
\int_{(i\RR)^m}\Gamma(
\Psi
)\Gamma(s - A_\rI^{-1}A_\rII Dw)\Gamma(Dw)
 \cdot e^{-i\<A_\rI s, \theta\>} c_0^{
-\Psi
}c_\rI^{-s + A_\rI^{-1}A_\rII Dw}c_\rII^{-Dw}\, dw\\
& = \frac{c_0^{|s|-t}}{e^{i\<A_\rI s, \theta\>}c_\rI^{s}} \int_{(i\RR)^m}\Gamma(
\Psi
)\Gamma(s - A_\rI^{-1}A_\rII Dw)\Gamma(Dw)\,\frac{dw}{x^w}.
\end{align*}
The bounds in the proof of Theorem~\ref{thm:expgro} implies that we can apply the Mellin inversion formula, which yields the equality
\[
|\det(D)|\,L^\theta(c)=2\pi i\,e^{-i\<A_\rI s,\theta\>}\Gamma(t)\,|\det(A_\rI)|\,M_f^\theta (c).
\]
Combining this with Lemma~\ref{lem:EM-MB} provides 
the desired result. 
\end{proof}

\begin{cor}\label{cor:LopsidedIntegralIndependence}
If $\beta$ is totally nonresonant for $A$, then, when viewed as analytic germs at some $c\in\CC^r\setminus\Sigma_A$, the $A$-hypergeometric functions given by meromorphic continuations of Euler--Mellin integrals arising from the distinct components of 
$\TT^n\setminus \overline{\cD_f}$  
are linearly independent.
\end{cor}
\begin{proof}
Let $\theta_1, \dots, \theta_k$ be representatives for the components of $\TT^n\setminus \overline{\cD_f}$. 
Suppose that $\ell_1, \dots, \ell_k$ are constants providing a vanishing linear combination of  $M_f^{\theta_1}(c), \dots, M_f^{\theta_k}(c)$, so that 
\[
g_B\sum_{j=1}^k \ell_j e^{i\<s,\theta_j\>}\,L^{\theta_j}(c)
= 2\pi i\,\Gamma(t)\,g_A\,\sum_{j=1}^k \ell_j\,M_f^{\theta_j} (c) = 0.
\]
It then follows from Theorem~\ref{thm:beukers}.\ref{item3} that $\ell_1 = \dots = \ell_k = 0$.
\end{proof}

In~\cite{beukers-monodromy}, it is noted that it is not always possible to construct a basis of the solution space of $H_A(\beta)$ only considering Mellin--Barnes integrals of the form \eqref{eqn:MB}. By Corollary~\ref{cor:LopsidedIntegralIndependence}, the same is true for constructing bases of solutions from 
meromorphic continuations of Euler--Mellin integrals arising only from the set of components of 
$\TT^n\setminus\overline{\cD_f}$. 
However, 
$\TT^n\setminus\overline{\cA'_f}$ has in general has more connected components than $\TT^n\setminus\overline{\cD_f}$, and in many cases it is possible to use Euler--Mellin integrals to construct a basis of solutions even though Mellin--Barnes integrals do not suffice. 

\begin{example}
Consider the matrix 
$A = \left[\begin{array}{cccc}
1 & 1 & 1 & 1\\
0 & 2 & 3 & 6\end{array}\right]$.
The $A$-discriminant is 
\begin{align*}
D_A(c) \,= & \, 1024 c_3 c_1^6-108 c_2^4 c_1^3+13824 c_0^2 c_3^2 c_1^3+8640 c_0 c_2^2
   c_3 c_1^3\\& \, -729 c_0 c_2^6+46656 c_0^4 c_3^3-34992 c_0^3 c_2^2 c_3^2+8748 c_0^2 c_2^4 c_3,
\end{align*}
whose coamoeba covers $\TT^4$. 
The maximal number of points in $\inter(\bZ_B)\cap (\Arg(c)B + 2\pi \ZZ[B])$ is five by~\cite{NP}, and hence
there is no Mellin--Barnes basis of solutions for $H_A(\beta)$ for any $\beta$. However, for a generic choice of coefficients $c$ the coamoeba of $f(z) = c_0 + c_1 z^2 +c_2 z^3 + c_3 z^6$ has six components in its complement, and for nonresonant $\beta$ these provide a Euler--Mellin basis of solutions for $H_A(\beta)$ by Proposition~\ref{prop:indep}.
\end{example}

We note that there exist matrices $A$ whose associated coamoeba has fewer than the desired $\vol(A)$-many number of connected components in the complement of its closure for all choices of coefficients $c$~\cite{jensthesis}. 
On the other hand, if $A$ is a circuit, then 
when $\beta$ is totally nonresonant, there always exists a Mellin--Barnes, and hence an extended Euler--Mellin, basis of integral representations for solutions of $H_A(\beta)$~\cite{lopsided}.

\section{$A$-hypergeometric rank-jumping example}
\label{sec:0134}
We conclude with an example first studied in \cite{0134}, 
where it was first seen that some parameters $\beta$ admit a higher-dimensional solution 
space for $H_A(\beta)$ than expected. 
We illustrate how meromorphic continuations of Euler--Mellin integrals capture these extra solutions at non-generic parameters $\beta$, 
offering a new tool to understand how these special functions arise. 

Consider the system $H_A(\beta)$ given by 
$A = \left[\begin{matrix} 1&1&1&1\\ 0&1&3&4\end{matrix}\right]$ 
and the unique 
parameter $\beta = (1,2)$ which yields a solution space whose dimension is one larger than expected. 
For this $A$, the Euler--Mellin integral is
\be\label{eq:0134 em int}
 M_f^{\Theta} (s,t,c) = \int_{\Arg^{-1}(\theta)} \frac{z^s}{(c_1 + c_2 z + c_3 z^3 + c_4 z^4)^t} \frac{d z}{z} 
\ee
for the polynomial $f(z) = c_1 + c_2 z + c_3 z^3 + c_4 z^4$ and 
$\theta\in\Theta$
for a fixed connected component $\Theta$ of 
$\TT^2\setminus \overline{\cA'_f}$. 
In order to calculate the corresponding $\Phi_f^{\Theta}$, we first expand \eqref{eq:0134 em int} 
five times in different directions, so that it converges for $(s,t) = (-2,-1)$. 
Upon expansion, $M_f^{\Theta} (s,t,c)$ is equal to 
\begin{align}\label{eq:exp 0134 em}
 \frac{(t)_2}{s}\int \frac{z^s h_1(z)}{f(z)^{t+2}} \frac{d z}{z} +
\frac{(t)_3}{s}\int \frac{z^s h_2(z)}{f(z)^{t+3}} \frac{d z}{z}
+ \frac{(t)_4}{s}\int \frac{z^s h_3(z)}{f(z)^{t+4}} \frac{d z}{z}  +
\frac{(t)_5}{s}\int \frac{z^s h_4(z)}{f(z)^{t+5}} \frac{d z}{z},
\end{align}
where all integrals are taken over $\Arg^{-1}(\theta)$ and 
$(t)_n = \Gamma(t+n)/\Gamma(t)$ is the Pochhammer symbol. 
This shows that when $(s,t) = (-2,-1)$, 
the entire function $\Phi_f^{\Theta}$ falls into the 
situation noted in Remark~\ref{rm:onePolyZeros}, 
and we thus ignore the factor $(t+1)$ in \eqref{eq:exp 0134 em}. 
To be explicit, 
{
\small 
\begin{align*}
h_1(z)  = &\ \frac{3 c_2 c_3 z^4}{s+1}+\frac{3 c_2 c_3 z^4}{s+3}+\frac{4 c_2 c_4 z^5}{s+1}+\frac{4 c_2 c_4 z^5}{s+4},\displaybreak[0]\\\
h_2(z)  = &\  \frac{36 c_1 c_3^2 z^6}{(s+3) (4t-s+2)}+\frac{48 c_1 c_3 c_4 z^7}{(s+3)(4t-s+1)}+\frac{48 c_1c_3 c_4 z^7}{(s+4) (4t-s+1)}
+\frac{64 c_1 c_4^2 z^8}{(s+4) (4 t-s)}
\\ &\quad 
+\frac{c_2^3 z^3}{(s+1) (s+2)}+\frac{3c_2^2 c_3 z^5}{(s+1) (s+2)}+\frac{4 c_2^2 c_4 z^6}{(s+1) (s+2)}
+\frac{27 c_2 c_3^2 z^7}{(s+3) (4t-s+2)} 
\\ &\quad 
+\frac{36 c_2 c_3 c_4 z^8}{(s+3) (4t-s+1)}+\frac{36 c_2 c_3 c_4 z^8}{(s+4) (4t-s+1)} 
+\frac{48c_2 c_4^2 z^9}{(s+4) (4 t-s)}
\\ &\quad 
+\frac{9 c_3^3 z^9}{(s+3) (4t-s+2)},\displaybreak[0]\\\ 
h_3(z)  = &\  \frac{48 c_1 c_3^2 c_4 z^{10}}{(s+3) (4t-s+1) (4t-s+2)}+\frac{48 c_1 c_3^2 c_4 z^{10}}{(s+4) (4t-s+1) (4t-s+12)}\\
&\quad  +\frac{64 c_1 c_3 c_4^2 z^{11}}{(s+4) (4t-s+1)^2}+\frac{36 c_2 c_3^2 c_4 z^{11}}{(s+3) (4t-s+1) (-s+4 t+2)} \\
&\quad  +\frac{36 c_2 c_3^2 c_4 z^{11}}{(s+4) (4t-s+1) (4t-s+2)}+\frac{48 c_2 c_3 c_4^2 z^{12}}{(s+4) (4 t-s) (4t-s+1)} \\
&\quad  +\frac{12 c_3^3 c_4 z^{13}}{(s+3) (4t-s+1) (4t-s+2)}+\frac{12c_3^3 c_4 z^{13}}{(s+4) (4t-s+1) (4t-s+2)},\displaybreak[0]\\\
\text{and}\quad 
h_4(z)  = &\ \frac{64 c_1 c_3^2 c_4^2 z^{14}}{(s+4) (4 t-s) (4t-s+1) (4t-s+2)} 
+\frac{48 c_2 c_3^2 c_4^2 z^{15}}{(s+4) (4 t-s) (4t-s+1) (4t-s+2)} \\
&\quad  +\frac{16 c_3^3 c_4^2 z^{17}}{s (s+4) (4 t-s) (4t-s+1) (4t-s+2)}.
\end{align*}
}
Each term in \eqref{eq:exp 0134 em} corresponds 
to a translation of the original integral \eqref{eq:0134 em int} 
and converges at $(s,t) = (-2,-1)$. 
In addition, the lack of a degree 2 term in $f$ means that no term of any 
$h_i(t)$ has both $(s+2)$ and $(4t -s +2)$ as factors in its denominator. 
Thus there are entire functions $\Phi_1$, $\Phi_2$, and $\Phi_3$  
in $s$ and $t$ such that 
\[
 \Phi_f^{\Theta} = (4t -s +2) \Phi_1 + (s+2)\Phi_2 +  (s+2)(4t -s +2)\Phi_3. 
\]
From this expression we see that while $\Phi^{\Theta}(-2,-1,c) = 0$ independently of $c$ and $\Theta$, we also obtain two functions 
$\Phi_1$ and $\Phi_2$ that are also solutions of $H_A(\beta)$. 
Explicit calculation reveals that
\[
 \Phi^\Theta_1(-2,-1,c) = 2 \,\frac{c_2^2}{c_1}
 \qquad \text{ and } \qquad 
 \Phi^\Theta_2(-2,-1,c) =  2 \,\frac{c_3^2}{c_4},
\]
for any choice of $\Theta$.  
These span the Laurent series solutions of the system $H_A(1,2)$, 
which has dimension two only at this parameter \cite{CDD}. 
The vanishing of $\Phi^\Theta$ at $\beta = (1,2)$ together with the appearance 
of $\Phi_1$ and $\Phi_2$ illustrates the first direct relationship between 
the computation of the local cohomology of the commutative ring 
$\CC[\partial_c] / \< \square_u\mid Au=0\>$ with respect to $\<\partial_c\>$ 
and the Laurent polynomial solutions of $H_A(1,2)$. 

\raggedbottom
\def\cprime{$'$} \def\cprime{$'$}
\providecommand{\MR}{\relax\ifhmode\unskip\space\fi MR }
\providecommand{\MRhref}[2]{%
  \href{http://www.ams.org/mathscinet-getitem?mr=#1}{#2}
}
\providecommand{\href}[2]{#2}

\begin{thebibliography}{GKZ89}

\bibitem[Beu11a]{beukers-monodromy}
Frits~Beukers, 
Monodromy of A-hypergeometric systems (2011).  
{\tt arXiv:1101.0493}

\bibitem[Beu11b]{beukers-irred}
Frits~Beukers, 
Irreducibility of A-hypergeometric systems, 
{\it Indag. Math.} (N.S.) \textbf{21} (2011), no. 1-2, 30--39.
	
\bibitem[CDD99]{CDD}
Eduardo~Cattani, Carlos~D'Andrea, and Alicia~Dickenstein, 
The $A$-hypergeometric system associated with a monomial curve, 
{\it Duke Math. J}. \textbf{99} (1999), no. 2, 179--207.
	
\bibitem[Erd37]{E}
Artur Erd\'elyi, Beitrag zur Theorie der konfluenten hypergeometrischen Funktionen von mehreren Ver\"anderlichen, {\it Sitzungsber.~Akad.~Wiss.~Wien} {\bf 146} (1937) 431--467.

\bibitem[For12]{jensthesis}
Jens Forsg{\aa}rd,
On Hypersurface Coamoebas and Integral Representations of Hypergeometric Functions,
\textit{Licentiate Thesis}, Department of Mathematics, Stockholm University
(2012). 

\bibitem[FJ12]{lopsided}
Jens Forsg{\aa}rd and Petter Johansson,
On the order map for hypersurface coamoebas (2012). 
{\tt arXiv:1205.2014}

\bibitem[GKZ90]{GKZc}  
Israel~Gelfand, Mikhail~Kapranov, and Andrei~Zelevinsky, 
Generalized Euler integrals and $A$-hypergeometric functions, 
\textit{Adv.~Math.} \textbf{84} (1990), 255--271.

\bibitem[GKZ94]{GKZa} 
Israel~Gelfand, Mikhail~Kapranov, and Andrei~Zelevinsky, 
\textit{Discriminants, Resultants and Multidimensional Determinants}, 
Birkh\"auser Boston, Inc., Boston, MA, 1994.

\bibitem[Joh10]{J} Petter Johansson: Coamoebas, \textit{Licentiate Thesis}, Department of Mathematics, Stockholm University (2010).

\bibitem[Nil09]{N}
Lisa~Nilsson, Amoebas, Discriminants, and Hypergeometric Functions, 
\textit{Doctoral thesis}, 
Stockholm University (2009).

\bibitem[NP10]{NP}
Lisa~Nilsson and Mikael~Passare, 
Mellin transforms of multivariate rational functions (2010). 
{\tt arXiv:1010.5060}

\bibitem[NS11]{NS} 
Mounir Nisse and Frank Sottile, 
Non-Archimedean coamoebae, 
\textit{Proceedings of the 2011 Bellairs Workshop in Number Theory: Non-Archimedean and Tropical Geometry}, to appear 
(2011). 
{\tt arXiv:1110.1033}

\bibitem[SST00]{SST}
Mutsumi~Saito, Bernd~Sturmfels, and Nobuki~Takayama, 
{Gr\"obner {D}eformations of {H}ypergeometric {D}ifferential {E}quations}, 
\textit{Springer--Verlag}, Berlin, 2000.

\bibitem[SW10]{sw-irred nonres}
Mathias~Schulze and Uli~Walther, 
Resonance equals reducibility for $A$-hypergeometric systems, \textit{Algebra Number Theory}, to appear (2010).
{\tt arXiv:1009.3569}
	
\bibitem[ST98]{0134}
Bernd~Sturmfels and Nobuki~Takayama, 
Gr\"obner bases and hypergeometric functions, 
Gr\"obner bases and applications (Linz, 1998), 246--258, 
\textit{London Math. Soc. Lecture Note Ser}. \textbf{251}, 
Cambridge Univ. Press, Cambridge, 1998.

\end{thebibliography}
\end{document}